\documentclass{amsart}

\usepackage[latin1]{inputenc}
\usepackage{subfig}
\usepackage{amsmath,amssymb}
\usepackage{mathtools}
\usepackage{stmaryrd}
\usepackage[]{caption}
\usepackage{todonotes}
\usepackage{float}

\newtheorem{theoreme}{Theorem}[section]
\newtheorem{prop}[theoreme]{Proposition}

\theoremstyle{plain}
\newtheorem{definition}[theoreme]{Definition}
\newtheorem{rem}[theoreme]{Remark}

\usepackage{xspace}

\newcommand{\PPP}{PPP\textsuperscript{*}\xspace}
\newcommand{\PPPs}{PPPs\textsuperscript{*}\xspace}
\newcommand{\zn}{{z_{\bullet}}}
\newcommand{\zb}{{z_{\circ}}}
\newcommand{\Ab}{{\mathcal A_{\circ}}}
\newcommand{\An}{{\mathcal A_{\bullet}}}

\title{Periodic Parallelogram Polyominoes}
\author{Adrien Boussicault and Patxi Laborde-Zubieta}
\address{LaBRI - Universit\'e de Bordeaux, 351 cours de la Lib\'eration F-33405 Talence cedex}
\email{plaborde@labri.fr}
\email{boussica@labri.fr}

\begin{document}

\maketitle

\begin{abstract}
A periodic parallelogram polyomino is a parallelogram polyomino such that we
glue the first and the last column. In this work we extend a bijection between
ordered trees and parallelogram polyominoes in order to compute the generating
function of periodic parallelogram polyominoes with respect to the height, the
width and the intrinsic thickness, a new statistic unrelated to the existing
statistics on parallelogram polyominoes. Moreover we define a rotation over
periodic parallelogram polyominoes, which induces a partitioning in equivalent
classes called strips. We also compute the generating function of strips using
the theory of P\'olya.
\end{abstract}

\section{Introduction}
Convex polyominoes and parallelogram polyominoes are classical combinatorial
objects which have been studied by Delest and Viennot in \cite{vien}.
They are counted respectively by
$
f_{n+2} = (2n+11)4^n - 4(2n+1){2n\choose n}
$
and
$
c_n = \frac{1}{n+1} {2n\choose n}.
$

In this article, we are interested in a new family of parallelogram
polyominoes, the periodic parallelogram polyominoes (cf.
Definition~\ref{def:PPP}). They model infinite strips embedded in a cylinder.
In order to study these objects, we adapt a bijection, between parallelogram
polyominoes and forests of trees, introduced in \cite{BRS}. The statistics that
will be studied are the height, the width, the semi-perimeter and a new
statistic called the intrinsic thickness. We were able to get a generating
function with respect to these statistics. Our new statistic doesn't seem to be
related to the area, which is another important statistic of parallelogram
polyominoes (cf. article~\cite{bousquet}). However, if we fix the intrinsic
thickness, the number of parallelogram polyominoes of semi-perimeter equal to
$n$ is equal to the sum of the areas under all the dyck paths of semi-length
equal to $n$. In this paper, we also give the generating function of periodic
parallelogram polyominoes up to rotation of columns.

We start by giving (Section~\ref{sec:PPP}) the definition of a periodic
parallelogram polyomino, we will denote them PPPs. Then, in a second section
(Section~\ref{sec:foret}), we will study the internal structure of a PPP, by
constructing for each PPP a graph called \emph{ordered cyclic forest}. In the
section Section~\ref{sec:emondage}, we define a reversible operation over
ordered cyclic forests called \emph{pruning}, which will reduce the study of
PPPs to trunk PPPs equipped with a marked ordered cyclic forest. We introduce
in Section~\ref{sec:rotation}, the notion of strip by defining a rotation on
PPPs, then, using trunk PPPs and ordered cyclic forests we decompose a strip
as a cycle of 4-tuples. This bijective decomposition allow us to compute the
generating function of strips (Section~\ref{sec:gen_bande}) and the generating
function of PPPs (Section~\ref{sec:gen_ppp}).

\section{Periodic Parallelogram Polyominoes (PPPs)}
\label{sec:PPP}

We start by giving a definition of \emph{parallelogram polyominoes}.
\begin{definition}[PP]
    A \emph{parallelogram polyomino} is a maximal set of cells of $\mathbb
    Z\times\mathbb Z$, defined up translation, contained in between two paths
    with North and East steps that intersect only at their starting and ending
    points.
\end{definition}

The Figure~\ref{fig:PP} shows an example of a parallelogram polyomino.
\begin{figure}[ht]
    \begin{center}
        \vspace{30mm}
        \includegraphics[scale=.5]{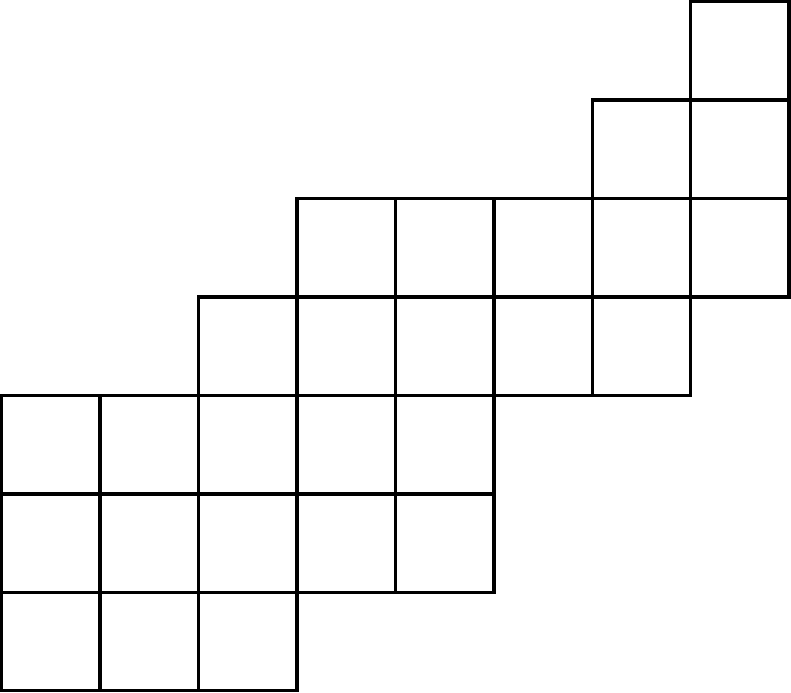}
    \end{center}
    \caption{ A parallelogram polyomino. } \label{fig:PP}
\end{figure}

The \emph{first} (resp. \emph{last}) \emph{column} corresponds to the leftmost (resp. rightmost) column.

\begin{definition}[PPP]
\label{def:PPP}

A \emph{periodic parallelogram polyomino} is a $PP$ with an admissible marked
cell in the first column. A cell is called admissible if its height is less or
equal to the size of the last column.
\end{definition}

This marking indicates the location where the first and the last column are
glued in order to obtain a \emph{periodic strip}. The Figure~\ref{fig:PPP}
gives an example of PPP and the beginning of the induced infinite strip. In the
rest of the article, multiple rows joint together will count as one. For
example, in the Figure~\ref{fig:PPP} the second row and the last row, starting
from the top, count as one.
Following this convention, we define the height of a PPP as its number of rows,
it is also the number of rows above the marked one. The number of column
defines the width of the PPP and the semi-perimeter is equal to the sum of the
width and the height. For example, the polyomino of Figure~\ref{fig:PPP} is of
height 5, width 8 and semi-perimeter 13.
\begin{figure}[ht]
    \begin{center}
        \includegraphics[scale=.4]{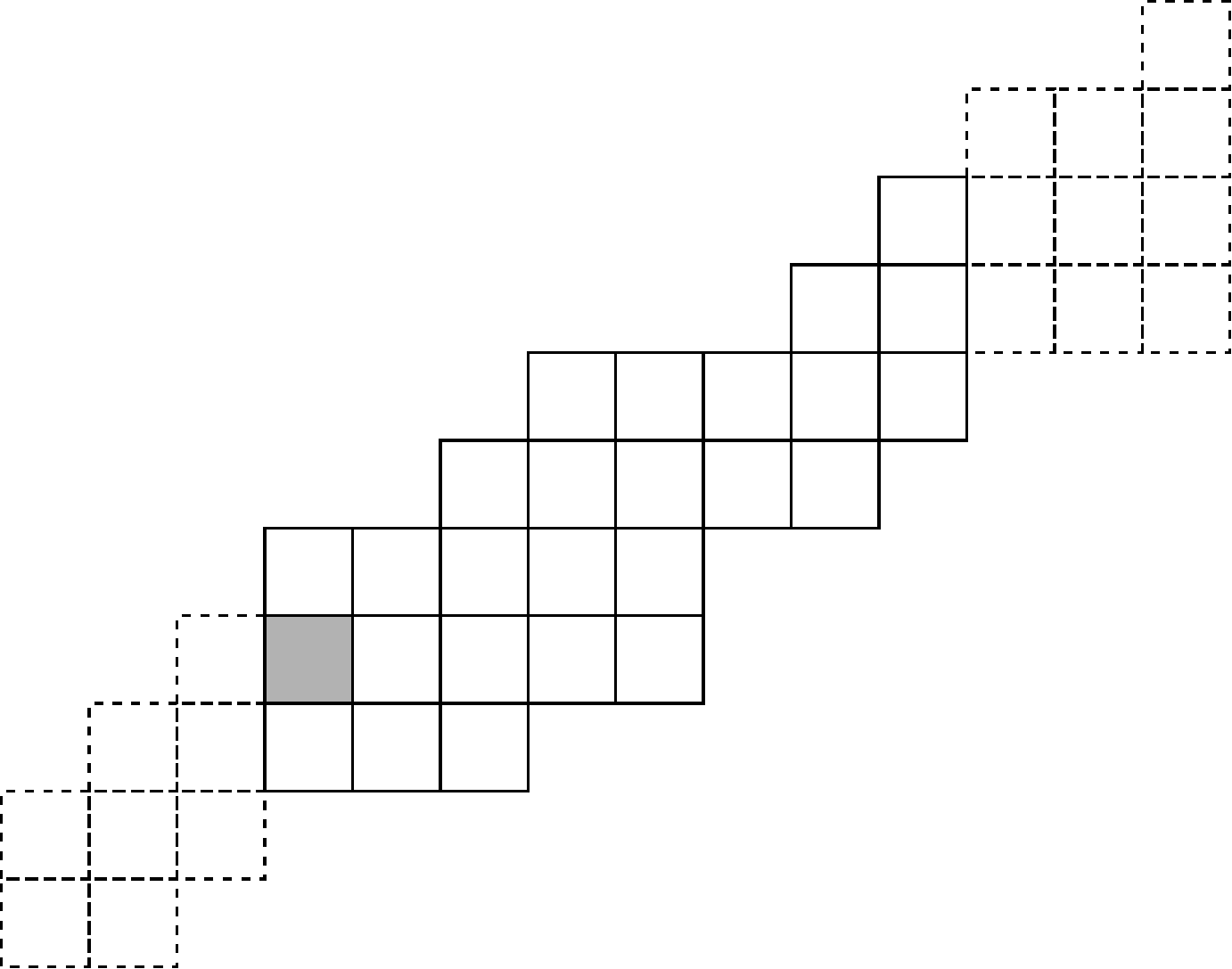}
    \end{center}
    \caption{ A periodic parallelogram polyomino. } \label{fig:PPP}
\end{figure}

From Section~\ref{sec:foret}, the study does not hold for \emph{degenerated} PPPs.

\begin{definition}
\label{def:PPPS}
A \emph{degenerated periodic parallelogram polyomino} is a PPP of rectangular
shape such that the marked cell is the topmost cell of the first column.
\end{definition}

For example, the left PPP of Figure~\ref{fig:PPP_degenere} is degenerated while
the right one is not.
\begin{figure}[ht]
    \begin{center}
        \begin{tabular}{c}
        \includegraphics[scale=.4]{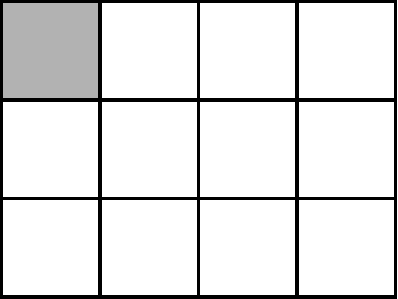} \\
        Degenerated.
        \end{tabular}
        \hspace{1cm}
        \begin{tabular}{c}
        \includegraphics[scale=.4]{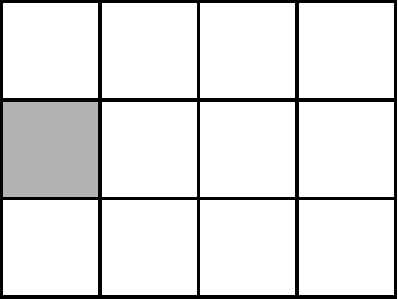} \\
        Non degenerated
        \end{tabular}
    \end{center}
    \caption{ A degenerated PPP and a non degerated one. } \label{fig:PPP_degenere}
\end{figure}

The generating function of degenerated PPPs, is equal to
$\frac{t}{1-t}\cdot\frac{y}{1-y}$, where $t$ counts the intrinsic thickness and
$y$ width, it therefore remains to study the non degenerated case. In the rest
of the article, a non degenerated PPP will be denoted by \PPP.

\section{Ordered cyclic forest of a \PPP}
\label{sec:foret}

In this section we define the notion of the ordered cyclic forest of a \PPP. It
is a graph which encodes the internal structure of a \PPP and gives a simpler
description of it.

Let $S$ be a finite set of elements and $father$ a map from $S$ to $S$.
Let $G_{father}$ be the graph whose vertices are the elements of $S$ and the
adjacency map $sons:=father^{-1}$. Each connected component of $G_{father}$
contains exactly one cycle. Indeed, by the Euler formula, the cyclomatic number
of each connected component is one since there are as much as edges as
vertices. Moreover, if $s$ is in a connected component $C$, the subgraph
$(father^k(s))_{k\ge 0}$ contains a cycle, and so do $C$.

For each vertex $s$ of $G_{father}$, we define a total order on $sons(s)$
denoted by $\overrightarrow{sons}(s)$. The graphs obtained by this construction
will be called \emph{ordered cyclic forests}. An example is given in the right
part of Figure~\ref{bij_phi}. In this figure, the sons are horizontally aligned
and ordered from left to right (from the biggest to the smallest). In order to
lighten the figure, the orientation of the edges is not explicitly written,
since it is not so important in our study.

We will now define a map $\Phi$ from $\PPPs$ to ordered cyclic forests. It
is based on a bijection between parallelogram polyominoes and forests of ordered
trees, introduced in \cite{BRS}. Let $P$ be a $\PPP$, the ordered cyclic forest
$\Phi(P)$ is obtained by defining $father$ and $\overrightarrow{sons}$ as
follows (the fact that multiple joint rows count as one still holds) :
\begin{itemize}
    \item the vertices are the topmost cells of columns and rightmost cells of rows;
    \item for all vertex $s$, if $s$ is the topmost cell of a column (resp.
        rightmost cell of a row) then $father(s)$ is the rightmost cell (resp.
        topmost cell) of its row (resp. column);
    \item let $s$ be a vertex and, $f_1$ and $f_2$ two of its sons. $f_1$ is
        bigger than $f_2$ if $f_1$ is at the right or above $f_2$ in $P$.
\end{itemize}
An example of $\Phi$ is given in Figure~\ref{bij_phi}, the vertices correspond
to the pointed cells.

\begin{figure}[ht]
    \begin{center}
    \includegraphics[scale=.35]{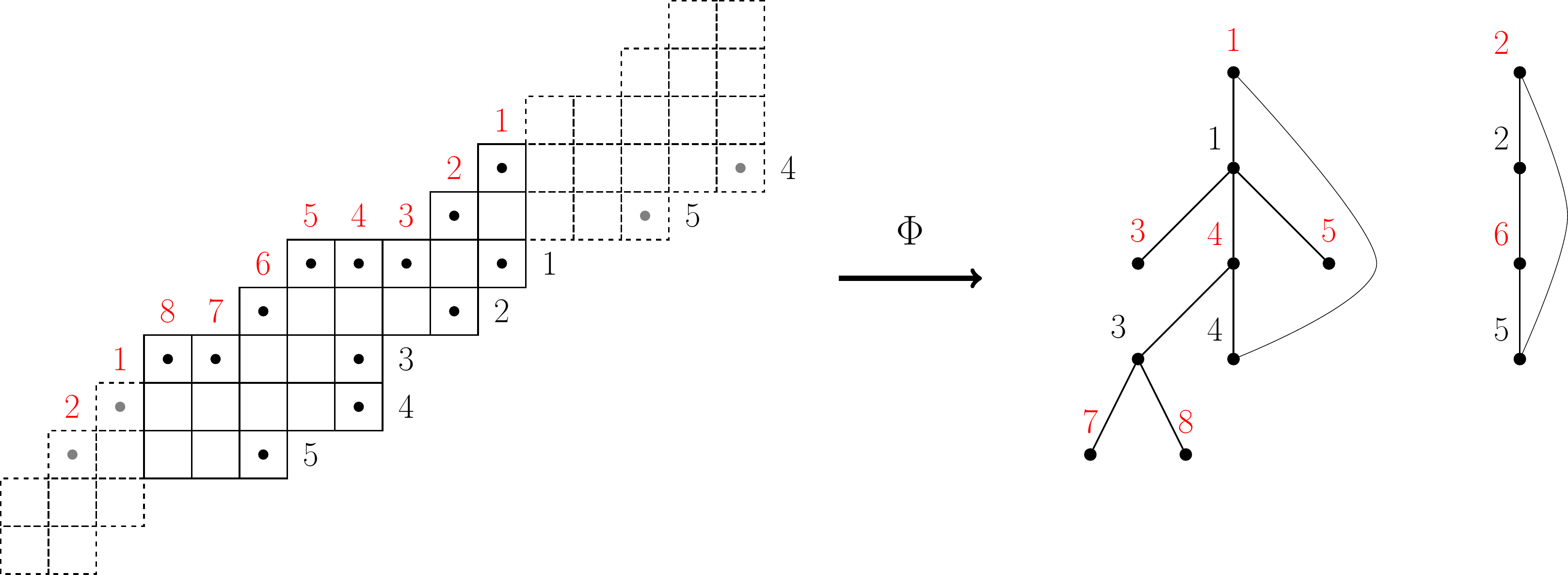}
    \end{center}
    \caption{The map $\Phi$.}
    \label{bij_phi}
\end{figure}

\section{Pruning and trunk \PPP}
\label{sec:emondage}

In the previous section, we introduced ordered cyclic forests. We will define
now an operation consisting in deleting recursively the rows and the columns in
$P$ corresponding to leaves in $\Phi(P)$. We call this operation, pruning. The
idea is to reduce the study of \PPPs to the ones whose ordered cyclic graph is
a disjoint union of cycles.

The leaves in $\Phi(P)$ correspond to rows and columns of $P$ which have only
one pointed cell. Hence, deleting a leaf in $\Phi(P)$ is equivalent to deleting
the corresponding row or column in $P$. If we know the position of the leaf
among its brotherhood we can construct back the corresponding row or column, up
to a rotation of column. For example, in Figure~\ref{figSuppresionColonne}, if
we want to reconstruct the leaf $\color{red}7$ we have no choice, but in the
case of the leaf $\color{red}8$, we don't know if it was in the first position
or in the last position. Therefore, pruning is reversible in the sens that we
can reconstruct a PPP up to rotation of columns. In order to, recover $P$, we
need to mark the vertex of $\Phi(P)$ corresponding to the first column of $P$.

\begin{figure}[ht]
    \begin{center}
    \includegraphics[scale=.35]{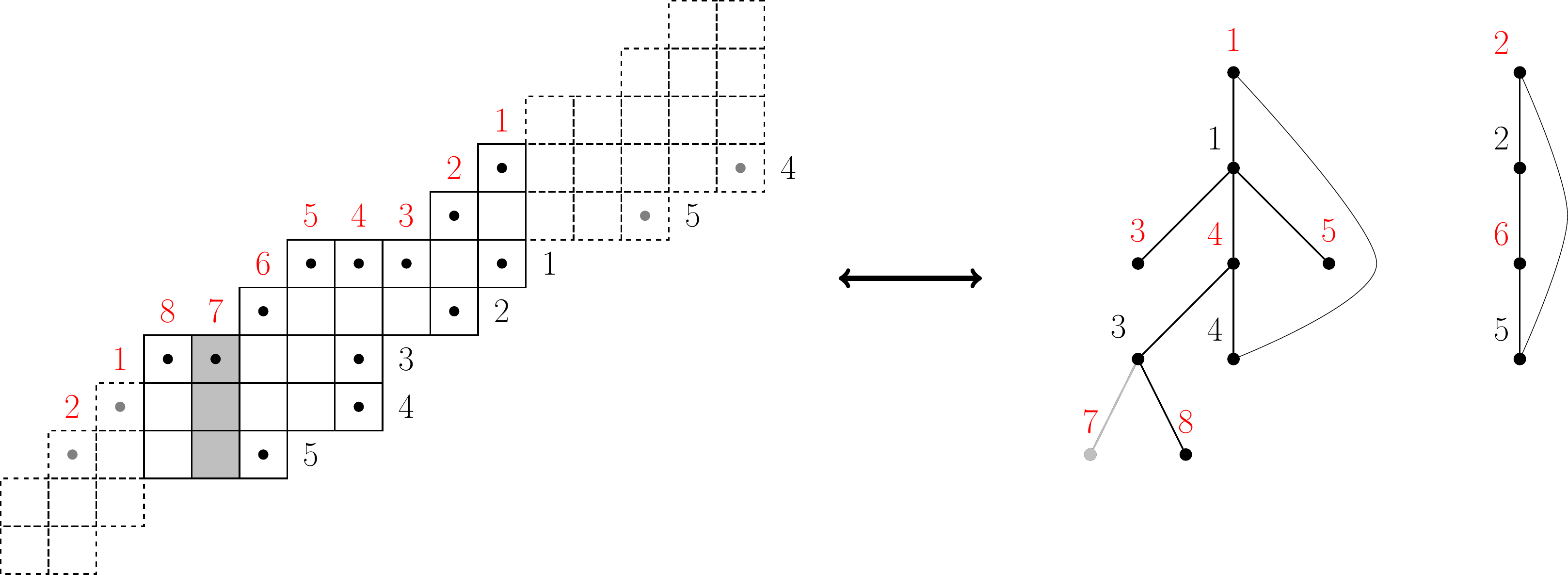}
    \end{center}
    \caption{
        Deletion of the leaf $\color{red}7$ of Figure~\ref{bij_phi}.
    }
    \label{figSuppresionColonne}
\end{figure}

After doing the complete pruning to $P$, the graph becomes a disjoint union of
cycles, we denote by $trunk(P)$ the \PPP we obtain. As the pruning is
invertible, if we know $\Phi(P)$, $trunk(P)$ and the marked vertex in $\Phi(P)$,
we can reconstruct $P$. Hence, we need to characterise the trunk \PPPs.

\begin{prop}\label{PropPrimitifs}
    A \PPPs is a trunk \PPP if:
\begin{itemize}
    \item the upper path is of the form $N^k(NE)^l$,
    \item the lower path is of the form $(EN)^lN^k$ and,
    \item the marked cell in the first column, is the topmost cell,
\end{itemize}
with $l$ and $k$ two positive integers.
\end{prop}
The Figure~\ref{fig:PPP_min} shows the trunk \PPP of the \PPP in Figure~\ref{figSuppresionColonne}. The corresponding integers $k$ and $l$ are respectively 2 and 4.

\begin{figure}[ht]
    \begin{center}
        \includegraphics[scale=.4]{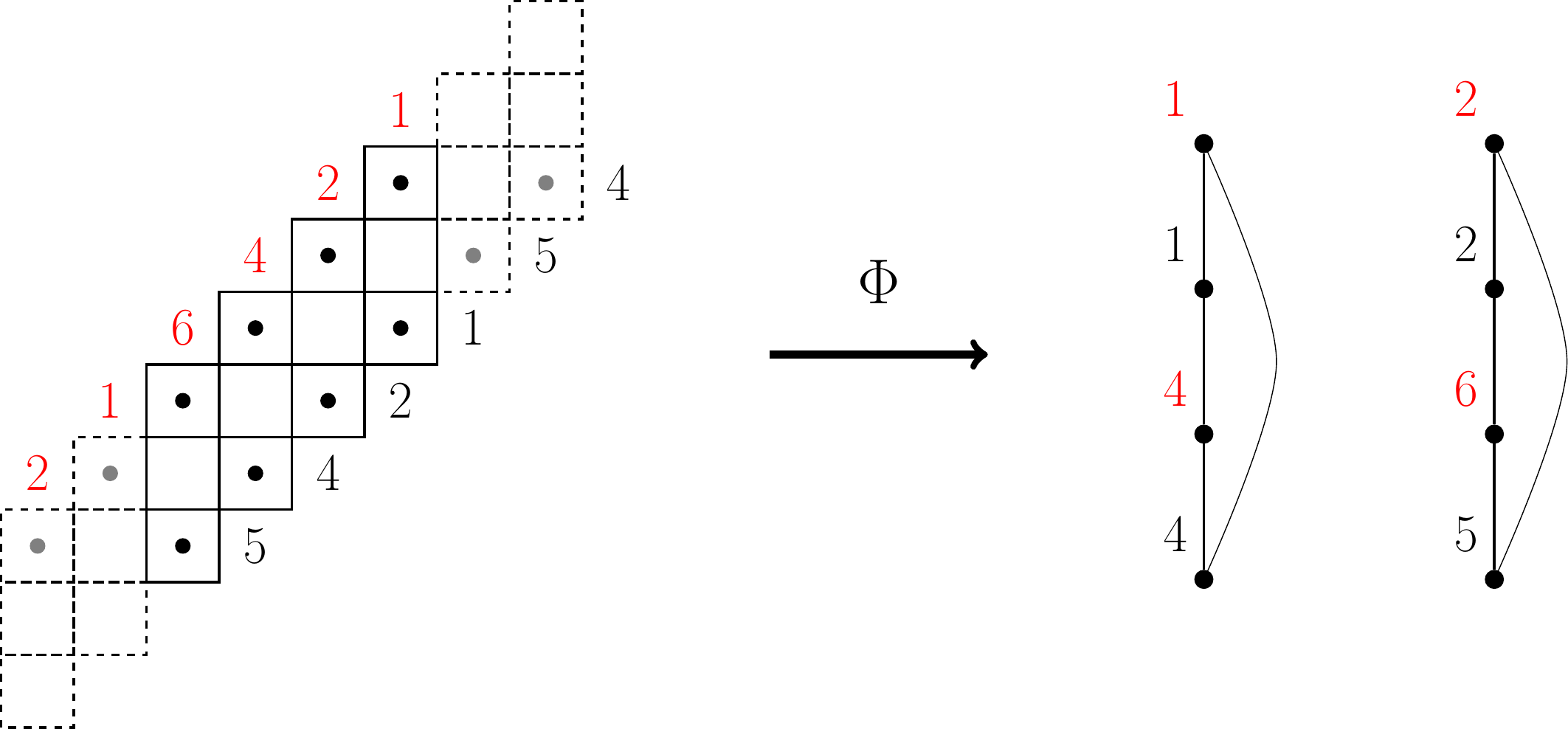}
    \end{center}
    \caption{The trunk \PPP of the \PPP of Figure~\ref{figSuppresionColonne} and its ordered cyclic forest .}\label{fig:PPP_min}
\end{figure}

\begin{prop}\label{PropCyclesNbSommets}
Let $P$ be a $\PPP$, then the cycles of the ordered cyclic forest $\Phi(P)$ are
of the same size, and this size is even.
\end{prop}
\begin{proof}
    The number of vertex of the cycles does not change when we prune the
    $\PPP$, and the Proposition~\ref{PropCyclesNbSommets} holds for trunk
    $\PPPs$.
\end{proof}

\begin{rem}
    It should be noted that multiple trunk \PPPs can give the same ordered
    cyclic forest. Indeed, $\Phi(P)$ contains $\frac{l}{\gcd(k,l)}$ disjoint
    cycles of size $2\gcd(k,l)$.
\end{rem}

Let $P$ be a \PPP and $k$ be the integer of $trunk(P)$ of
Proposition~\ref{PropPrimitifs}. The integer $k$ is called the \emph{intrinsic
thickness} of $P$.

\section{Action of the rotation over the \PPPs : the strips}
\label{sec:rotation}

In this section, we will define strips by defining how we rotate a \PPP.

If we join recursively a copy of a \PPP in the first and in the last column, we
obtain an infinite strip. Two different \PPPs can generate the same infinite
strip, in particular, they will have the same ordered cyclic forest. In order
to define more formally the notion of strip, we need to define the rotation of
a \PPP. Let $P$ be a \PPP of width at least 2. If we mark the second column at
the level of the topmost cell of the first column, and we move the first column
to the last position while respecting the initial marking. We obtain a \PPP
that we note $r(P)$. In particular, $\Phi(p)$ and $\Phi(r(P))$ give the same
ordered cyclic forest. This construction defines an automorphism $r$ of \PPPs
called \emph{rotation}. Let us denote $G$ the group generated by $r$.

\begin{definition}
    The rotation induces an action of the group $G$ on the set of \PPPs. The
    orbits of this action will be called \emph{strips}. Since the
    semi-perimeter and the intrinsic thickness are invariant by rotation, we
    extend the two statistics to strips.
\end{definition}
This notion of strip, model the embedding of an infinite strip in a cylinder
whose circumference is give by the semi-perimeter. That is why, if two \PPPs of
different semi-perimeter give graphically the same infinite strip by
translation, their strips will be viewed as different.

\section{Generating function of strips}
\label{sec:gen_bande}

In order to get the generating function of strips, we describe a bijection
between strips and pairs $(k,c)$ made of a positive integer $k$, that will
correspond to the intrinsic thickness, and a cycle $c$ of 4-tuples of ordered
trees. Then, we will use the theory of P\'olya to compute the generating
function.

\begin{prop}\label{Prop:BijBandes}
    The strips are in bijection with pairs made of:
    \begin{itemize}
        \item a positive integer $k$
        \item a cycle of size $l$ of 4-tuples of ordered trees.
    \end{itemize}
\end{prop}
\begin{proof}[Proof (Sketched)]
    Let $P$ be a \PPP. $trunk(P)$ is uniquely
    define with the positive integers $k$ and $l$ of
    Proposition~\ref{PropCyclesNbSommets}. Each column of $trunk(P)$ contains
    two consecutive vertices $v_1$ and $v_2$ of the same cycle of $\Phi(P)$.
    Hence, we associate to each column, the 4-tuple of ordered trees (two for
    each vertex) rooted in $s_1$ and $s_2$. Following the order of the columns
    in $trunk(P)$, we construct a list of 4-tuples. Instead of a list we will
    consider a cycle since the pruning is reversible up to rotation. We will
    denote by $\psi(P)$ the pair $(k,c)$.
\end{proof}

The set of cycles composed of 4-tuples of ordered trees will be denoted by $C$.
The intrinsic thickness and the semi-perimeter of a strip $b$ are respectively
denoted $it(b)$ and $sp(b)$.

We want to compute the generating function $B(z,t)$ of strips, with $z$
counting the semi-perimeter and $t$ the intrinsic thickness:
\[
B(z,t)
:=
\sum_{b \mathrm{\ strip}}
    t^{it(b)}
    z^{sp(b)}
=
\sum_{k \ge 1 \atop c \in C}
    t^k
    z^{
        sp( \Psi^{-1} (k,c) )
    }
=
\frac{t}{1-t} \cdot
\sum_{c \in C}
    w(c)
.
\]
with
$
w(c) = 
z^{
    sp( \Psi^{-1} ((1,c)) )
}
$.

In order to find the generating function of strips, we just need to count the cycles whose elements are colored with 4-tuples of ordered trees.

The theory of P\`olya (cf. \cite{bergeron,polya}) gives a formula to compute
the generating function of cycles $c=(c_1, c_2, \dots, c_n)$ with the weight
$w(c) = \prod_i{w(\mathrm{\ color}(c_i))}$ and colored with colors taken in a set
$A$:
\[
Z(A)
:=
\sum_{c \mathrm{\ cycle}} w(c)
=
-\sum_{i\geqslant1}
\frac{\varphi(i)}{i}\log\left(1- \sum_{a \in A} w(a)^i\right),
\]
with $\varphi(i) := |\{ k<i \  | \   k \mathrm{\ and\ } i \mathrm{\  are\  relatively\ 
prime.}\}|$ is the Euler phi function.
Let $w(a) = z^{|a|}$ and $\mathcal{G}(z)$ be the generating function
$\sum_{a \in A} w(a)$
of $A$, then
$
\mathcal{G}(z^i)
$
is equal to
$
\sum_{a \in A} (w(a))^i.
$
Hence, the generating function of cycles becomes
\begin{equation}
-\sum_{i\geqslant1}\frac{\varphi(i)}{i}\log(1-\mathcal G(z^i)).
\end{equation}
In our case, we want to compute the generating function of strips with respect to the semi-perimeter. Since the semi-perimeter of a strip $b$ is equal to the number of vertices in its underlying graph $\Phi(P)$, hence we need to count the number of non-root vertices in $\Psi(P)$ plus 2 vertices for each 4-tuples in the cycle. That's why, the set $A$ corresponds to the set of 4-tuples of ordered trees with the weight:
\[w((T_1,T_2,T_3,T_4))=z^{|(T_1,T_2,T_3,T_4)|},\]
with $|(T_1,T_2,T_3,T_4)| = |T_1| + |T_3| + |T_3| + |T_4| + 2$ and $|T|$ is the number of non-root vertices of $T$. All that remains is to compute the generating function $\mathcal G(z)$ of the set $A$, which is equal to $z^2\mathcal A(z)^4$, where $\mathcal A(z)$ is the generating function of ordered trees counted with respect to the number of non-root vertices. It satisfies the equation:
\[\mathcal A(z) = \frac{1}{1-z\mathcal A(z)},\]
which can be solved:
\[\mathcal A(z) = \frac{1-\sqrt{1-4z}}{2z}.\]
Finally, the generating function of strips $\mathcal B(z,t)$ where $z$ counts the semi-perimeter and $t$ the intrinsic thickness, is equal to: 
\begin{equation}
    B(z,t)
    =-\frac{t}{1-t} \cdot \sum_{i\geqslant1}\frac{\varphi(i)}{i}\log(1-\frac{(1-\sqrt{1-4z})^4}{16z^{2i}}).
\end{equation}

This formula is quite surprising, because the intrinsic thickness is a statistic independent from other statistics. Hence, it is sufficient to study the family of strips with intrinsic thickness equal to 1 to characterise completely the combinatory of the entire family.

\section{Generating function of PPPs}
\label{sec:gen_ppp}

In this section we will give a non symmetric version of Proposition~\ref{Prop:BijBandes} in order to completely encode a \PPP with a pair $(k,s)$ with $k$ being the intrinsic thickness and $s$ a sequence of 4-tuples of ordered trees such that the first 4-tuple has a marked vertex.

In order to extend the Proposition~\ref{Prop:BijBandes} to \PPPs, it is
sufficient to mark the vertex corresponding the top cell of the first column.
Hence, we color in black the vertices corresponding to the top cells of columns
and in white the rightmost cells of rows. We obtain two types of bicolored
trees, the ones with black vertices at odd height and the ones with black
vertices at even height. In particular, each 4-tuple
$(T_{i,1},T_{i,2},T_{i,3},T_{i,4})$ is composed of two black rooted trees, and
two white rooted trees. Marking the top cell of a column correspond to marking
a non-root black vertex or the two black roots of a same 4-tuple. If we put in
first position the marked 4-tuple, we get a sequence instead of a cycle.

\begin{prop}\label{Prop:BijPPP}
    The \PPPs are in bijection with pairs composed of a positive integer $k$
    (the intrinsic thickness) and a non empty $l$-tuple of 4-tuples of
    bicolored ordered trees such that:
    \begin{itemize}
        \item each 4-tuple is composed of two black rooted trees and two white rooted trees,
        \item in the first 4-tuple, we mark a non-root black vertex or the two black roots.
    \end{itemize}
\end{prop}

To compute the generating function of \PPPs, we need to find the generating functions of black rooted and white rooted ordered trees. We denote them respectively $\An(\zn,\zb)$ and $\Ab(\zn,\zb)$,
with $\zn$ counting the black vertices and $\zb$ the white ones. Those two generating functions satisfies the following equations.
\[
\An(\zn,\zb) = \frac{1}{1-\zb\Ab} \mathrm{\ and\ } \Ab(\zn,\zb) = \frac{1}{1-\zn\An}.
\]
After solving them, we obtain
\begin{equation}
\An(\zn,\zb) = \frac{\zn-\zb+1-\sqrt{(\zn-\zb+1)^2-4\zn}}{2\zn} = \Ab(\zb,\zn)
\end{equation}
Hence, the generating function of a 4-tuple of ordered trees is
\begin{equation}
\zn\zb\An(\zn,\zb)^2\Ab(\zn,\zb)^2.
\end{equation}
More over, if we mark one of the black vertices, we get
\begin{equation}
    \zn\partial_\zn\zn\zb\An(\zn,\zb)^2\Ab(\zn,\zb)^2
\end{equation}
Finally, the generating function of the \PPPs with $\zn$ counting the width,
$\zb$ the height and $t$ the intrinsic thickness is equal to
\begin{equation}
    P(\zn,\zb,t) =
    \frac{t}{1-t}
    \cdot
    \frac{
        \zn\partial_\zn\zn\zb\An(\zn,\zb)^2\Ab(\zn,\zb)^2
    }{
        1-\zn\zb\An(\zn,\zb)^2\Ab(\zn,\zb)^2
    }.
\end{equation}

As in the previous section, we notice that the intrinsic thickness is a
statistic independent from the other statistics. Hence, in order to
characterize combinatorially the \PPPs, it is sufficient to study the \PPPs of
intrinsic thickness 1. The generating function of \PPPs of intrinsic height 1
is equal to:
\[
P(z,z,t)[t]
=
1\cdot z + 6 \cdot z^2 + 29\cdot z^3 + 130\cdot z^4 + 562 \cdot z^5 + 2380\cdot z^6 + \ldots
\]
The sequence of the previous coefficients appears in OEIS (cf.~\cite{oeis}), it
corresponds to the sequence A008549. The $n$-th term is the sum of the areas
under all the dyck paths of semi-length $n$.

\vspace{1cm}
\section*{Acknowledgement}
\label{sec:ack}

The authors are grateful to Sre\v{c}ko Brlek for his attention, his advices and
for the numerous discussions we had during our stays in the LaCIM (Laboratoire
de Combinatoire et d'Informatique Math\'ematique) at the University of Quebec at
Montreal.

We thank the LIRCO (Laboratoire International Franco-Qu\'eb\'ecois de Recherche en
Combinatoire) for financing our respective stays at LaCIM.

This research was driven by computer exploration using the open-source software
\texttt{Sage}~\cite{sage} and its algebraic combinatorics features developed by
the \texttt{Sage-Combinat} community \cite{Sage-Combinat}.

\bibliographystyle{plain}
\bibliography{bibliographie}
\end{document}